%% file: StratifyingPP.tex
\documentclass[11pt,letterpaper]{article}
\usepackage{amsfonts, amsmath, amssymb, amscd, amsthm, color, graphicx, mathrsfs, wasysym, setspace, mdwlist, calc, float, xcolor, tocloft, hyperref}

\hypersetup{colorlinks,hypertexnames=false,
    linkcolor={red!50!black},
    citecolor={blue!80!black},
    urlcolor={blue!80!black}}

\let\OLDthebibliography\thebibliography
\renewcommand\thebibliography[1]{
  \OLDthebibliography{#1}
  \setlength{\parskip}{1.5pt}
  \setlength{\itemsep}{1.5pt plus 0.3ex}
}

\newcommand{\PP}{\mathrm{PP}}

\newcommand{\e}{\varepsilon}
\newcommand{\NN}{\mathbb{N}}
\newcommand{\ZZ}{\mathbb{Z}}
\renewcommand{\wr}{\textrm{\,wr\,}}
\renewcommand{\d}{{\rm d}}

\newcommand{\la}{\langle}
\newcommand{\ra}{\rangle}
\newcommand{\Ker}{\operatorname{Ker}}

\newcommand{\RR}{\mathbb{R}}
\newcommand{\CC}{\mathbb{C}}

\newcommand{\one}{\mathbf{1}}

\newcommand{\CGHX}{Cay(G,X\sqcup \mathcal H)}

\newtheorem{thm}{Theorem}[section]
\newtheorem{cor}[thm]{Corollary}
\newtheorem{lem}[thm]{Lemma}
\newtheorem{prop}[thm]{Proposition}

\theoremstyle{definition}
\newtheorem{defn}[thm]{Definition}

\theoremstyle{remark}

\newtheorem{ex}[thm]{Example}

\title{A paradoxical route from hyperbolic geometry to proper proximality}
\author{D. Osin, K. Toyosawa, Z. Yang}
\date{}

\begin{document}
\maketitle

\begin{abstract}
For every integer $n\ge 2$, we introduce a new combinatorial condition $\PP(n)$ inspired by paradoxical decompositions of groups. As $n$ increases, the property $\PP(n)$ weakens, and the resulting hierarchy interpolates between group-theoretic manifestations of negative curvature and proper proximality. More precisely, we prove that a countable group is properly proximal if and only if it satisfies $\PP(n)$ for some $n$. On the other hand, every acylindrically hyperbolic group satisfies $\PP(2)$; furthermore, finitely generated $\PP(2)$-groups admit a geometric characterization: they are precisely the groups containing strongly quasi-convex, non-cyclic, free subgroups. In particular, we obtain that every countable acylindrically hyperbolic group is properly proximal. Finally, for every $n\in\NN$, we provide examples of finitely generated properly proximal groups that do not satisfy $\PP(n)$, thus showing that our hierarchy is indeed infinite.
\end{abstract}

\section{Introduction}

The notion of proper proximality, motivated by rigidity questions for von Neumann algebras, was introduced by Boutonnet, Ioana, and Peterson in \cite{BIP}. Every free ergodic probability measure preserving action of a countable group $G$ on a probability space $(X,\mu)$ gives rise to a group measure space factor $L^\infty(X)\rtimes G$, with a distinguished Cartan subalgebra
$L^\infty(X)$. A central problem is to determine when this subalgebra is uniquely determined, up to unitary conjugacy, by the ambient factor. Proper proximality provides a useful condition in this direction: if $G$ is properly proximal, the group von Neumann algebra $LG$ has no weakly compact Cartan subalgebra, and every weakly compact Cartan subalgebra of $L^\infty(X)\rtimes G$ is unitarily conjugate to $L^\infty(X)$ \cite{BIP}. For further developments and a survey of related results and open problems, we refer to \cite{DKEP23,IPR,Pet}.

Although motivated by applications to operator algebras, proper proximality is closely related to geometric and dynamical properties of groups. The original examples in \cite{BIP} include non-amenable bi-exact groups, non-elementary convergence groups, and lattices in non-compact semi-simple algebraic groups over local fields. Further examples were obtained from non-positive curvature and hierarchicaly hyperbolic geometry \cite{HHL23}, actions on trees and wreath products \cite{DKE}, positive first $\ell^2$-Betti numbers for exact groups \cite{Din24}, and small cancellation theory \cite{Oya26}.

Many of these examples belong to the class of acylindrically hyperbolic groups introduced in \cite{Osi16}. Recall that a group is \textit{acylindrically hyperbolic} if it admits a non-elementary action on a hyperbolic space satisfying an additional condition, which can be viewed as a weak form of properness for the induced action on the space of distinct pairs (for the precise definition, see Section 3.1). This class includes non-elementary hyperbolic and relatively hyperbolic groups, all but finitely many mapping class groups of punctured closed surfaces \cite{Bow}, $\mathrm{Out}(F_n)$ for $n\ge 2$ \cite{BF}, automorphism groups of non-elementary hyperbolic and finitely generated infinitely-ended groups \cite{GH}, most $3$-manifold groups \cite{MO,MO19}, groups of deficiency at least $2$ \cite{Osi15}, small cancellation groups \cite{GS}, and many other examples (see \cite{Osi18} and references therein). 

The question of whether every acylindrically hyperbolic group is properly proximal was open until now; for a discussion of this problem and some of its particular instances, see \cite{BIP,DKE,Oya26,Pet}. The main goal of our paper is to provide an affirmative answer for countable groups and, more generally, to understand the relation between proper proximality and manifestations of negative curvature in group theory.

Our approach is inspired by the classical characterization of nonamenability in terms of paradoxical decompositions. More precisely, we introduce the following condition combining paradoxical behavior with controlled expansion under right translations. For a subset $X$ of a set $Y$, we denote by $\one_X\colon Y\to \{ 0, 1\}$ the indicator function of $X$.

\begin{defn}\label{Def:PP}
A pair of subsets $A\subseteq B$ of a group $G$ is called
\emph{admissible} if $|Ag\setminus B|<\infty $ for all $g\in G$.
We say that $G$ satisfies $\PP(n)$ if there are admissible pairs
$A_i\subseteq B_i$ and elements $g_i\in G$, $1\le i\le n$, such that
\begin{equation}\label{Eq:PP}
 \sum_{i=1}^n\one_{g_iA_i}(x)
 \ \ge\ 1+\sum_{i=1}^n\one_{B_i}(x)
 \qquad \forall\, x\in G.
\end{equation}
The \emph{PP-complexity} of $G$ is defined to be the minimal $n$ such that $G$ satisfies $\PP(n)$, or $\infty$ if no such $n$ exists.
\end{defn}

\begin{ex}\label{Ex:Free} The free group with basis $\{ a, b\}$ satisfies $\PP(2)$. This can be seen in several ways. 

\begin{enumerate}
    \item[(a)] We can take $A_1=B_1$ (respectively, $A_2=B_2$) to be the set of all reduced words beginning with $a$ (respectively, $b$), and let $g_1=a^{-1}$, $g_2=b^{-1}$. The verification of the required properties is straightforward. For finitely generated groups, it is not difficult to show that the existence of equal sets $A_i=B_i$ satisfying Definition \ref{Def:PP} is equivalent to having infinitely many ends.
    \item[(b)] A slightly more sophisticated choice is illustrated in Fig. \ref{Fig0}: we can take $A_1$ (respectively, $A_2$) to be the set of all reduced words beginning with $a^2$ (respectively, $b^2$),  $B_1$ (respectively, $B_2$) to be the set of all reduced words beginning with $a$ (respectively, $b$), and let $g_1=a^{-2}$, $g_2=b^{-2}$. 
\end{enumerate}
\end{ex}

The latter example is especially instructive: it is used in the proof of the geometric characterization of $\PP(2)$-groups (see Theorem \ref{Thm:AH} (b)), and essentially the same idea appears in the proof of $\PP(2)$ for  acylindrically hyperbolic groups.

Our first result provides a characterization of proper proximality in terms of $\PP(n)$.

\begin{thm}\label{Thm:PP}
A countable group $G$ is properly proximal if and only if it satisfies $\PP(n)$ for some $n\ge 2$. 
\end{thm}

The next theorem shows that the strongest condition in our hierarchy, $\PP(2)$, is satisfied by all acylindrically hyperbolic groups. Furthermore, in the class of finitely generated groups, it has a precise geometric meaning. 

\begin{thm}\label{Thm:AH}
\begin{enumerate}
    \item[(a)] Every acylindrically hyperbolic group satisfies $\PP(2)$. In particular, every countable acylindrically hyperbolic group is properly proximal. 
    \item[(b)] A finitely generated group satisfies $\PP(2)$ if and only if it contains a strongly quasi-convex, non-cyclic, free subgroup.
\end{enumerate}
\end{thm}

\begin{figure}
\hspace{10mm}\scalebox{.8}{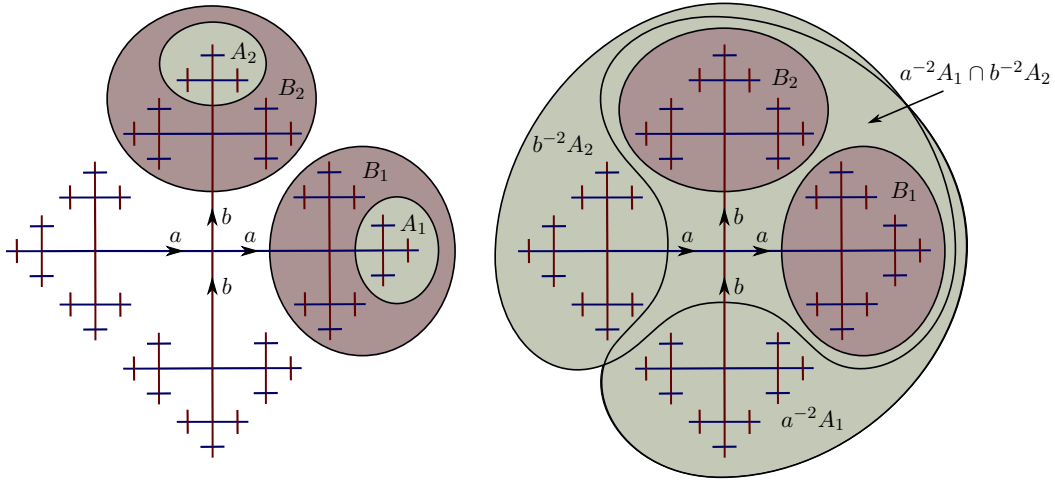}\\
  \caption{Condition $\PP(2)$ for the free group $F(a,b)$}\label{Fig0}
\end{figure}

Recall that a subgroup $H$ of a finitely generated group $G$ is \textit{strongly quasi-convex} if quasi-geodesics in $G$ with endpoints in $H$ stay uniformly close to $H$ (see Section 3.3 for details). This property generalizes the Morse lemma for geodesics in hyperbolic spaces. The combination of the main result of \cite{Osi16}, \cite[Theorem 2.24]{DGO}, and \cite[Theorem~2]{Sis} implies that every finitely generated acylindrically hyperbolic group contains a strongly quasi-convex, non-cyclic, free subgroup. However, the converse implication does not hold. Indeed, the recent paper \cite{AZ} provides a striking example of a finitely generated group $G$ such that every action of $G$ on a hyperbolic space is degenerate (specifically, elliptic or parabolic) while $G$ is Morse local-to-global and contains a Morse element. By \cite[Corollary I]{RST}, every such group contains a strongly quasi-convex, non-cyclic, free subgroup. 

One consequence of Theorem \ref{Thm:AH} (b) is that a finitely generated $\PP(2)$-group cannot decompose as a direct product of two infinite subgroups. Indeed, a finitely generated direct product of two infinite groups is wide in the terminology of \cite{DS}, and therefore cannot have cut points in asymptotic cones (see Example (1) following Definition 3.4 in \cite{BDM}), let alone Morse elements and strongly quasi-convex free subgroups. This observation implies that the PP-complexity of $F_2\times F_2$ is $3$. (The upper bound is provided by an elementary combinatorial argument.) We do not have explicit examples of groups of PP-complexity $n>3$. However,  we show that for every $n\in \NN$, there exists a properly proximal group of PP-complexity at least $n$. More precisely, we prove the following.

\begin{thm}\label{Thm:FI}
There exists a finitely generated, properly proximal group $G$ and a sequence of finite index  subgroups $H_i\le G$ such that the PP-complexity of $H_i$ tends to $\infty$ as $i\to \infty$. 
\end{thm}

Note that the class of properly proximal groups is closed under commensurability \cite[Proposition 1.6]{BIP}; in particular, all subgroups $H_i$ are properly proximal and hence have finite PP-complexity. It is not difficult to show that the PP-complexity of a group is preserved by passing to quotients by finite normal subgroups and extensions with finite kernel. However, Theorem~\ref{Thm:FI} implies the following.

\begin{cor}
   The PP-complexity is not invariant under passing to subgroups of finite index. 
\end{cor}

The paper is organized as follows. Theorem \ref{Thm:PP} relating $\PP(n)$ to proper proximality is proved in the next section. In Section~3, we recall the necessary background on acylindrically hyperbolic groups and prove both parts of Theorem \ref{Thm:AH}. In the course of the proof, we obtain a version of the classical ping-pong lemma that allows us to construct strongly quasi-convex free subgroups of finitely generated groups (see Proposition \ref{Prop:ping}); this result appears to be of independent interest. Finally, in Section 4, we prove Theorem \ref{Thm:FI} and compute the PP-complexity of the direct product of two free groups.

\paragraph{Acknowledgments.} The first author has been supported by the NSF grant DMS-2405032. The second author was funded by the Deutsche Forschungsgemeinschaft (DFG, German Research Foundation) under Germany's Excellence Strategy EXC 2044/2 –390685587, Mathematics Münster: Dynamics–Geometry–Structure.

\paragraph{AI disclosure.} During the preparation of the first draft, which contained only a proof of proper proximality for acylindrically hyperbolic groups, discussions with ChatGPT 5.6 helped the second author explore possible approaches and identify several key references. The current version is based on a different mathematical approach, developed by the authors without the use of AI tools. In preparing this version, ChatGPT 6 was used solely for language editing, proofreading, and literature searches.

\section{Proper proximality and $\PP(n)$-groups}

In this paper, ``countable" always means ``countably infinite". The goal of this section is to prove Theorem \ref{Thm:PP}, which follows from the combination of Propositions \ref{Prop:prox-PP} and \ref{Prop:PP-prox}.

For a countable group $G$, let $\mathcal R_0(G)$ consist of the bounded functions $u\colon G\to\CC$ that coincide with their right translations modulo $c_0(G)$. That is, $u\in \ell^\infty(G)$ belongs to $\mathcal R_0(G)$ if and only if for all $\e >0$ and all $g\in G$, there exists a finite set $F\subseteq G$ such that $|u(xg)-u(x)|<\e$ for all $x\in G\setminus F$. Clearly, left translations
$(\lambda_gu)(x)=u(g^{-1}x)$ preserve this $C^*$-algebra and we have
$$
 \mathcal R_0(G)/c_0(G)
   =\big(\ell^\infty(G)/c_0(G)\big)^{G_r}
$$
in the notation of \cite{BIP}. By \cite[Theorem~4.3(iii)]{BIP}, a countably infinite group $G$ is properly proximal if and only if this quotient admits no $G$-invariant state.

\paragraph{2.1. From proper proximality to $\PP(n)$.}
Let $\ell^\infty(G,\RR)$ (respectively, $\mathcal R_0(G,\RR)$) be the subset of real-valued functions of $\ell^\infty(G)$ (respectively, $\mathbb R_0(G)$). Recall that a \emph{mean} on a subspace $V$ of $\ell^\infty (G, \RR)$ containing $\one_G$ is a positive linear functional $m\colon V\to \RR$ with $m(\one_G)=1$. Positivity gives $\inf u\le m(u)\le\sup u$ for every $u\in V$, so $m$ is automatically
continuous. The following characterization of properly proximal groups is well-known to experts (see, for example \cite[p. 2858]{DKEP23}). We include the proof for the convenience of the reader.

\begin{lem}\label{Lem:mean}
A countable group $G$ is properly proximal if and only if there is no $G$-invariant mean on $\mathcal R_0(G,\RR)$. 
\end{lem}

\begin{proof}
We write $[u]$ for the image of $u\in \mathcal R_0(G,\RR)$ in $\mathcal R_0(G)/c_0(G)$. For any state $\phi$ on $\mathcal R_0(G)/c_0(G)$, the linear functional $m\colon \mathcal R_0(G, \RR)\to \RR$ defined by $m(u)=\phi([u])$ is a mean. 

Conversely, suppose that we have a left-invariant mean $m$ on $\mathcal R_0(G,\RR )$. Since $G$ is infinite, $m$ vanishes on $c_0(G;\RR)=c_0(G)\cap \ell^\infty (G,\RR )$; indeed, left-invariance implies $m(\one_{g})=m(\one_{1})$ for all $g\in G$. The normalization condition $m(\one_G)=1$ forces $m(\one_g)=0$ for all $g\in G$, and continuity implies vanishing on $c_0(G;\RR)$. Hence, the map $\phi\colon \mathcal R_0(G)/c_0(G)\to \mathbb C$ given by $$\phi([u+iv])=m(u)+im(v)$$ for any $u,v\in \mathcal R_0(G, \RR)$ is well-defined. It is straightforward to verify that $\phi $ is a state. 

Thus, the required equivalence follows immediately from \cite[Theorem~4.3(iii)]{BIP}.
\end{proof}

We say that a function $u\in \ell^\infty(G,\RR)$ is \textit{uniformly positive} if $\inf_G u >0$. The proof of following lemma makes use of the classical argument going back to the theory of (non)amenable groups, see also \cite[Theorem~4.3]{BIP}.

\begin{lem}\label{Lem:HB}
Suppose a countable group $G$ is properly proximal. Then there exist $u_1,\ldots,u_k\in \mathcal R_0(G,\RR)$ and $g_1,\ldots,g_k\in G$ such that
$\sum_{i=1}^k(\lambda_{g_i}u_i- u_i)$ is uniformly positive.
\end{lem}

\begin{proof}
Consider the subspace 
$$
W= \left\{ \left. \sum_{i=1}^\ell (\lambda_{f_i}v_i-v_i)\; \right| \; v_i\in \mathcal R_0(G,\RR), \; f_i\in G,\; \ell\in \NN\right\}.
$$ 
Suppose $W$ contains no uniformly positive function. Then $\sup_G w\ge 0$ for every $w\in W$, as otherwise $-w$ would be uniformly positive. It follows that
$W\cap\RR\one_G=\{0\}$. Hence, the linear functional $$m_0(c\one_G+w)=c,\qquad  \forall\, c\in\RR,\; \forall\, w\in W$$ is well defined and dominated by $v\mapsto \sup_G v$, since $\sup_G(c\one_G+w)=c+\sup_G w\ge c$. By the Hahn--Banach theorem, $m_0$ can be extended to a linear functional
$m$ on $\mathcal R_0(G,\RR)$ such that $m(v)\le\sup_G v$. Applying this inequality to $-v$ gives $\inf_G v\le m(v)$. Thus $m$ is a mean on $\mathcal R_0(G,\RR)$, and its vanishing on $W$ makes it left-invariant. This contradicts  Lemma~\ref{Lem:mean}.
\end{proof}

\begin{prop}\label{Prop:prox-PP}
Every countable, properly proximal group satisfies $\PP(n)$ for some $n\in \NN$.
\end{prop}

\begin{proof}
Let $G$ be a countable, properly proximal group and let $u_1,\ldots,u_k\in \mathcal R_0(G,\RR)$ and $g_1,\ldots,g_k\in G$ be the data provided by Lemma \ref{Lem:HB}. Replacing each $u_i$ with $(u_i+c)/2c$ for a sufficiently large $c\in \RR$, we can assume that $0\le u_i(x)\le 1$ for all $x\in G$. Let 
$$
\e= \inf_G\left(\sum_{i=1}^k (\lambda_{g_i}u_i-u_i) \right) >0.
$$
Let $L=\lceil 2k/\e\rceil$. For every $1\le i\le k$ and $1\le j\le L$, we define
$$
A_{ij}=\{ x\in G\mid Lu_i(x)\ge j\}\qquad {\rm and}\qquad B_{ij}=\{ x\in G\mid Lu_i(x)\ge j-1\}.
$$
Clearly, $A_{ij}\subseteq B_{ij}$. Further, if $x\in A_{ij}$ and $xg\notin B_{ij}$ for some $g\in G$, then 
\begin{equation}\label{Eq:ui}
u_i(x)-u_i(xg)\ge j/L - (j-1)/L=1/L.
\end{equation}
Since $u_i(x)-u_i(xg)\in c_0(G)$, only finitely many $x$ can satisfy (\ref{Eq:ui}) for every fixed $g\in G$. Thus, each pair $A_{ij}\subseteq B_{ij}$ is admissible.

Recall that $0\le u_i(x)\le 1$ by our assumption. Elementary counting implies that for any fixed $1\le i\le k$,  every element $x\in G$ is contained in more than 
$Lu_i(x)-1$ sets $A_{ij}$ and at most $Lu_i(x)+1$ sets $B_{ij}$. Therefore, we obtain
$$
   \sum_{i,j} (\one_{g_iA_{ij}}-\one_{B_{ij}}) > L \sum_{i=1}^k \Big(u_i(g_i^{-1}x)-u_i(x)\Big) - 2k\ge L\e -2k \ge 0.
$$
Since the left-hand side is integer-valued, it is at least $1$. Thus, $G$ satisfies $\PP(kL)$.
\end{proof}

\paragraph{2.2. From $\PP(n)$ to proper proximality.} We begin with an auxiliary result.

\begin{lem}\label{Lem:u}
For any admissible pair of subsets $A\subseteq B$ of a countable group $G$, there exists a function $u\colon G\to [0,1]$ such that $u\in\mathcal R_0(G)$ and $\one_A\le u\le\one_B$.
\end{lem}

\begin{proof}
The cases $A=\emptyset$ or $B=G$ are immediate. Otherwise, we can choose a proper left-invariant integer-valued metric $\d$ on $G$ and set
\begin{equation}\label{eq:u}
 u(x)=\frac{\d(x,G\setminus B)}
 {\d(x,A)+\d(x,G\setminus B)}.
\end{equation}
Note that for each $R>0$, the set
$$
 \{a\in A\mid \d(a,G\setminus B)\le R\}
 \subseteq\bigcup_{\d(1,g)\le R}(Ag\setminus B)g^{-1}
$$
is finite by our assumption. Properness of $\d$ then implies that there are only finitely many points within distance $R/2$ of both $A$ and $G\setminus B$. Therefore, the denominator in (\ref{eq:u}) tends to infinity as $x\to \infty$ in $G$. Since $|\d(x,S)-\d(xg,S)|\le \d(1,g)$ for every non-empty $S\subseteq G$, it follows that
$u(xg)-u(x)\to0$ as $x\to \infty$ for every fixed $g\in G$. Thus,
$u\in\mathcal R_0(G)$. The required pointwise inequalities
follow immediately from (\ref{eq:u}). 
\end{proof}

\begin{prop}\label{Prop:PP-prox}
If a countable group satisfies $\PP(n)$ for some $n$, then it is properly proximal.
\end{prop}
\begin{proof}
Let $G$ be a countable group, $A_i\subseteq B_i$ and $g_i\in G$, $1\le i\le n$, be as in Definition \ref{Def:PP}. For every $1\le i\le n$, let $\one_{A_i}\le u_i\le\one_{B_i}$ be the function provided by Lemma \ref{Lem:u}. Inequality (\ref{Eq:PP}) yields
$$
 \sum_{i=1}^n(\lambda_{g_i}u_i-u_i) \ge\sum_{i=1}^n(\one_{g_iA_i}-\one_{B_i})\ge\one_G.
$$
Assuming there is a left-invariant mean $m$ on $\mathcal R_0(G,\RR)$, we obtain
$$
1= m(\one_G)\le m\left(\sum_{i=1}^n(\lambda_{g_i}u_i-u_i)\right) = \sum_{i=1}^n\Big(m(\lambda_{g_i}u_i)-m(u_i)\Big)=0,
$$
a contradiction. Using Lemma \ref{Lem:mean}, we have concluded that $G$ is properly proximal.
\end{proof}

\section{$\PP(2)$ groups and negative curvature}

In this section, we discuss the relation between $\PP(2)$ and various manifestations of negative curvature in group theory. In particular, we give a proof of Theorem \ref{Thm:AH}. 

\paragraph{3.1. Background on hyperbolically embedded subgroups.}
The proof of the first claim of Theorem \ref{Thm:AH} makes use of the notion of hyperbolically embedded subgroups introduced in \cite{DGO}, following ideas in \cite{Osi06a,Osi06b}. We provide a very brief review of the necessary background and refer to \cite{DGO,Osi16,Osi18} for more details.

Let $G$ be a group, $\{H_i\}_{i\in I}$ a collection of subgroups of $G$, $X$ a subset of $G$ such that $G$ is generated by $X$ and $\bigcup_{i\in I}H_i$. We let $\mathcal H=\bigsqcup_{i\in I}H_i$ and regard $X\sqcup\mathcal H$ as a generating alphabet, distinguishing letters from different members of the union even when they represent the same element of $G$. For every $i\in I$, the Cayley graph $Cay(H_i,H_i)$ (where all elements of $H_i$ are taken as generators) can be thought of as a subgraph of $\CGHX$. Given $g,h\in H_i$, let $d_i(g,h)$ be the length of a shortest path in the Cayley graph $\CGHX$ connecting $g$ to $h$ and containing no edges of the subgraph $Cay(H_i,H_i)$; we set $d_i(g,h)=\infty$ if no such path exists.

\begin{defn}\label{Def:HE}
The collection $\{H_i\}_{i\in I}$ is \emph{hyperbolically embedded} in $G$ with respect to $X$, written $\{H_i\}_{i\in I}\hookrightarrow_h(G,X)$, if the following conditions hold.
\begin{enumerate}
\item[(a)] $\CGHX$ is hyperbolic.
\item[(b)] For every $i\in I$ and every $n\in \NN$, there are only finitely many elements $h\in H_i$ satisfying $d_i(1,h)\le n$.
\end{enumerate}
\end{defn}

Let $p$ be a path in $\CGHX$. An \emph{$H_i$-component} of $p$ is a maximal nontrivial subpath all of whose edges are labelled by letters from $H_i$. If $p$ is a cycle, maximality is understood in the cyclic sense. Two $H_i$-components of $p$ are \emph{connected} if their vertices lie in the same left coset of $H_i$. An $H_i$-component is \emph{isolated} in $p$ if it is not connected to any other $H_i$-component of $p$. Clearly, every $H_i$-component of a geodesic consists of a single edge and is isolated in that geodesic.

We use the following consequence of \cite[Proposition 4.14]{DGO}.

\begin{lem}\label{Lem:Omega}
Suppose that $\{H_i\}_{i\in I}\hookrightarrow_h(G,X)$. There exists $C>0$ such that, for every geodesic $n$-gon $p$ in $\CGHX$ and every isolated $H_i$-component $q$ of $p$, the element $h\in H_i$ represented by the label of $q$ satisfies $d_i(1,h)\le Cn$.
\end{lem}

All group actions on metric spaces considered in this paper are assumed to be isometric.

\begin{defn}
An action of a group $G$ on a metric space $S$ is \emph{acylindrical} if, for every $\varepsilon>0$, there exist $R,N>0$ such that, for any $x,y\in S$ with $d(x,y)\ge R$, at most $N$ elements $g\in G$ satisfy
$$
d(x,gx)\le\varepsilon\quad {\rm and}\quad d(y,gy)\le\varepsilon.
$$
A group $G$ is \emph{acylindrically hyperbolic} if it admits a non-elementary acylindrical action on a hyperbolic space; equivalently, $G$ is not virtually cyclic and admits an acylindrical action on a hyperbolic space with unbounded orbits (see \cite{Osi16}).
\end{defn}

Every acylindrically hyperbolic group $G$ has a unique maximal finite normal subgroup, called its \emph{finite radical} and denoted by $K(G)$ \cite[Theorem 2.24]{DGO}. We will need the following result, which is a particular case of \cite[Lemma 6.18]{DGO}.

\begin{lem}\label{Lem:HE}
Let $G$ be an acylindrically hyperbolic group with $K(G)=\{1\}$. For every $m\in \NN$, there exist infinite order elements $h_1,\ldots,h_m\in G$ and $X\subseteq G$ such that $\{\langle h_1\rangle,\ldots,\langle h_m\rangle\}\hookrightarrow_h (G,X)$.
\end{lem}

\paragraph{3.2. Proof of $\PP(2)$ for acylindrically hyperbolic groups.} For finitely generated groups, part (a) of Theorem \ref{Thm:AH} follows immediately from part (b) and the known fact that finitely generated acylindrically hyperbolic groups contain strongly quasi-convex, non-cyclic, free subgroups. For infinitely generated groups, the idea of pulling the paradoxical structure from free subgroups used in the proof of part (b) given below still works after small modification (specifically, the projection provided by \cite[Lemma 4.5]{AHO} should be used in place of the map constructed in Proposition \ref{Prop:contr}). However, we choose to give a direct and significantly simpler proof of part (a) of Theorem \ref{Thm:AH}, which is inspired by the proof of inner non-amenability of acylindrically hyperbolic groups in \cite[Theorem~8.14]{DGO}. 

We first deal with a special case; although the same idea works in the general case as well, the triviality of the finite radical makes the proof somewhat cleaner. 

\begin{figure}
\centering\scalebox{.9}{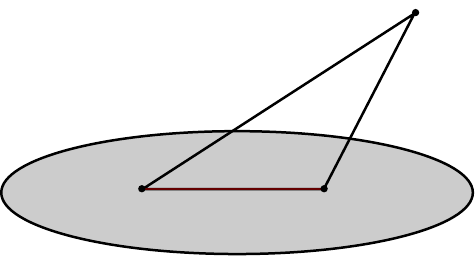}\\
  \caption{Bounding the $\d_i$-distance between nearest points}\label{Fig2}
\end{figure}

\begin{lem}\label{lem:AH}
Every acylindrically hyperbolic group with trivial finite radical satisfies $\PP(2)$.
\end{lem}

\begin{proof}
By Lemma \ref{Lem:HE}, there exist infinite cyclic subgroups $H_i=\langle h_i\rangle$, $i=1,2$, and a subset $X\subseteq G$ such that $\{H_1,H_2\}\hookrightarrow_h(G,X)$. Let $\d$ denote the word metric on $G$ with respect to the generating set $X\cup H_1\cup H_2$ and let $d_i$, $i=1,2$, be the metric on $H_i$ defined above. Let also $C$ be the constant provided by Lemma \ref{Lem:Omega} and let 
$$
K=\max \{ |m| \mid \exists \, i\in \{ 1,2\}\; {\rm such \; that\;}  d_i(1,h_i^m)\le 4C \} +1.
$$
Note that the maximum exists by Definition \ref{Def:HE} (b). The only purpose of adding $1$ is to ensure that $K$ is positive, which is necessary to avoid degeneracy in the construction below. 

Consider any $g\in G$. A geodesic from a nearest point of $H_i$ to $g$ in $\CGHX$ has no vertices in $H_i$ other than the initial vertex. If $h_i^r,h_i^s$ are distinct nearest points of $H_i$ to $g$, join both of them to $g$ by geodesics $p$ and $q$ in $\CGHX$ and connect them by an $H_i$-edge $e$ (see Fig. \ref{Fig2}). Since $e$ is isolated in the geodesic triangle with sides $p$, $q$, and $e$, we obtain $d_i(1,h_i^{r-s})\le 3C$ by Lemma~\ref{Lem:Omega}; consequently, $|r-s|\le K$. Thus the set of nearest points of $H_i$ to $g$ is finite. We can therefore define
$$
\pi_i(g)=\max\{n\in\ZZ\mid d(h_i^n,g)=d(H_i,g)\}.
$$
Left multiplication by $h_i^m$ preserves distances in $\CGHX$ and preserves $H_i$ setwise. Therefore,
\begin{equation}\label{Eq:pii}
\pi_i(h_i^m g)=m+\pi_i(g)
\qquad \forall \, m\in\ZZ\ \; \forall\, g\in G.
\end{equation}

We claim that 
\begin{equation}\label{Eq:min}
\min\{|\pi_1(g)|,|\pi_2(g)|\}\le K,\qquad \forall\, g\in G.
\end{equation} 
Indeed, fix a geodesic $p$ from $1$ to $g$. There exists $i\in \{ 1,2\}$ such that $p$ does not begin with an $H_i$-edge; since $p$ is geodesic, it has no $H_i$-component with vertices in the coset $H_i$. Let $a=h_i^{\pi_i(g)}$. We join $a$ to $g$ by a geodesic $q$ in $\CGHX$. If $a=1$, we have $\pi_i(g)=0$ and (\ref{Eq:min}) obviously holds. If $a\ne 1$, the $H_i$-edge $e$ from $1$ to $a$ is isolated in the geodesic triangle $eqp^{-1}$. Therefore $d_i(1,a)\le 3C$ by Lemma \ref{Lem:Omega} and $|\pi_i(g)|\le K$ by the choice of $K$, and (\ref{Eq:min}) follows.

\begin{figure}
\centering\scalebox{.9}{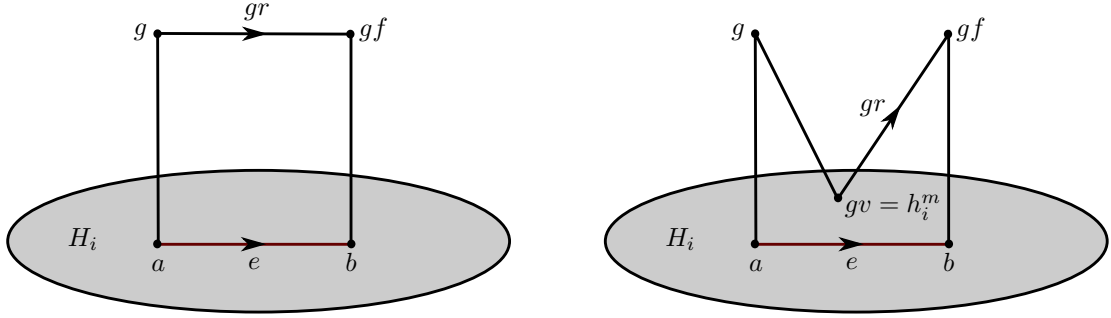}\\
  \caption{Two cases in the proof of admissibility of $A_i\subseteq B_i$}\label{Fig3}
\end{figure}

For $i=1,2$, we let
$$
A_i=\{g\in G\mid \pi_i(g)\ge3K\}\qquad {\rm and}\qquad B_i=\{g\in G\mid \pi_i(g)\ge2K\}.
$$
Clearly, $A_i\subseteq B_i$. We will first show that these pairs are admissible. To this end, fix $f\in G$, choose a geodesic $r$ from $1$ to $f$, and let $V$ denote its (finite) set of vertices. Suppose that $g\in A_i$ and $gf\notin B_i$. Then $\pi_i(g)\ge 3K$ and $\pi_i(gf)<2K$, so $\pi_i(g)-\pi_i(gf)>K$. Let $a=h_i^{\pi_i(g)}$ and $b=h_i^{\pi_i(gf)}$. Consider the geodesic quadrilateral with vertices $a$, $g$, $gf$, $b$, using the translate $gr$ as the side from $g$ to $gf$ and an $H_i$-edge as the side connecting $b$ to $a$. The sides from $a$ to $g$ and from $b$ to $gf$ have no vertices in $H_i$ other than $a$ and $b$, since $a$ and $b$ are nearest points of $H_i$ to $g$ and $gf$, respectively. If $gr$ had no vertex in $H_i$, the edge from $b$ to $a$ would be isolated. This would give $d_i(1,b^{-1}a)\le 4C$ and thus $|\pi_i(g)-\pi_i(gf)|\le K$, a contradiction. Therefore $gv\in H_i$ for some $v\in V$, and we have $g=h_i^m v^{-1}$ for some $m\in \ZZ$. By (\ref{Eq:pii}), the assumptions $g\in A_i$ and $gf\notin B_i$ imply
$$
3K-\pi_i(v^{-1})\le m<2K-\pi_i(v^{-1}f).
$$
For every fixed $v\in V$, only finitely many integers $m$ satisfy these inequalities. Since $V$ is finite, there are only finitely many $g\in A_i$ such that $gf\notin B_i$. Thus $A_if\setminus B_i$ is finite.

Finally, let $g_i=h_i^{-4K}$ for $i=1,2$. Using (\ref{Eq:pii}) again, we obtain  $$g_iA_i=\{g\in G\mid \pi_i(g)\ge-K\}.$$ The inequality (\ref{Eq:min}) implies $G=g_1A_1\cup g_2A_2$ and $B_1\cap B_2=\emptyset$. If $g\in B_i$, then $\pi_i(g)\ge2K$ and $|\pi_j(g)|\le K$ for $j\ne i$, so $g\in g_1A_1\cap g_2A_2$. Therefore
$$
\mathbf1_{g_1A_1}+\mathbf1_{g_2A_2} \ge1+\mathbf1_{B_1}+\mathbf1_{B_2}.
$$
Together with admissibility, this proves $\PP(2)$.
\end{proof}

\begin{proof}[Proof of Theorem \ref{Thm:AH} (a)]
Let $G$ be an acylindrically hyperbolic group. It is well-known and easy to prove that $G/K(G)$, where $K(G)$ is the finite radical of $G$, is acylindrically hyperbolic and has trivial finite radical (see, for example, \cite[Lemma~5.10]{Hull}). Let $\overline{A_i}\subseteq \overline{B_i}\subseteq G/K(G)$ and $\overline{g_i}\in G/K(G)$, $i=1,2$, be the admissible pairs and elements witnessing property $\PP(2)$ in $G/K(G)$. Let $A_i\subseteq B_i$ be the full preimages of $\overline{A_i}\subseteq \overline{B_i}$, and let $g_i$ be any preimage of $\overline{g_i}$ in $G$. The inverse images of the  admissible pairs are admissible since $K(G)$ is finite, and the inequality (\ref{Eq:PP}) is preserved under taking full preimages. Thus, $G$ satisfies $\PP(2)$.
\end{proof}

\paragraph{3.3. Strongly quasi-convex subgroups from ping-pong.} Given a (combinatorial) path $p$ in a graph $\Gamma$, we denote by $\ell(p)$ its length and by $p_-,p_+$ its endpoints. Let $\lambda\ge 1$ and $c\ge 0$. Recall that a (rectifiable) path $p$ in a metric space $(S,\d)$ is a \textit{$(\lambda,c)$-quasi-geodesic} if every subpath $q$ of $p$ 
satisfies $$\ell(q)\le \lambda \d(q_-,q_+)+c.$$ 
A subgroup $H$ of a finitely generated group $G$ is
\textit{strongly quasi-convex} if, for every $\lambda\ge 1$ and $c\ge 0$, there exists $R\ge 0$ such that every $(\lambda,c)$-quasi-geodesic in the Cayley graph of $G$ with endpoints in $H$ lies in the
$R$-neighborhood of $H$. This property is independent of the choice of a finite generating set of $G$ \cite{Tra}.

The following strengthening of the standard ping-pong lemma will be used in the proof of Theorem \ref{Thm:AH} (b) and appears to be of independent interest.

\begin{prop}\label{Prop:ping}
Let $G$ be a finitely generated group, $r\ge 1$ an integer, $x_1,\ldots,x_r\in G$. Suppose that there are
subsets $D_a\subseteq G$ indexed by letters in the formal alphabet
$\mathcal A=\{x_1^{\pm1},\ldots,x_r^{\pm1}\}$ such that the following conditions hold.
\begin{enumerate}
\item[(a)] The sets $D_a$, $a\in\mathcal A$, are pairwise disjoint and
$G\setminus\bigcup_{a\in\mathcal A}D_a$ is nonempty.
\item[(b)] $a(G\setminus D_{a^{-1}})=D_a$ for every $a\in\mathcal A$.
\item[(c)] $D_ag\cap D_b$ is finite for all $g\in G$ and distinct
$a,b\in\mathcal A$.
\end{enumerate}
Then $\{x_1,\ldots,x_r\}$ is a basis of a strongly quasi-convex free subgroup of $G$.
\end{prop}

We divide the proof into three lemmas, all of which are stated under the assumptions of Proposition \ref{Prop:ping}. Let $F=\la x_1,\ldots,x_r\ra$ and $Q=G\setminus\bigcup_{a\in\mathcal A}D_a$. Replacing every $D_a$ with $D_aq^{-1}$ for a fixed $q\in Q$ preserves conditions (a)--(c). Thus, we can assume that $1\in Q$ without loss of generality. For $g\in G$, let
$$
E_g=\bigcup_{\substack{a,b\in\mathcal A\\a\ne b}}(D_a\cap D_bg^{-1}).
$$
These sets are finite by (c).

\begin{lem}\label{Lem:FQ}
Elements $x_1,\ldots,x_r$ form a basis of a free subgroup $F\le G$ and
\begin{equation}\label{Eq:GU}
G=\bigsqcup\limits_{f\in F}fQ.
\end{equation}
\end{lem}
\begin{proof}
If $w=a_1\cdots a_n$ is a nonempty reduced word, induction on $n$ using (b) gives $wQ\subseteq D_{a_1}$. In particular, $w\ne1$ in $G$. Therefore, $x_1,\ldots,x_r$ freely generate a subgroup $F$, and the translates $fQ$, $f\in F$, are pairwise disjoint.

It remains to show that $G=FQ$. Arguing by contradiction, suppose that there exists $x\in G\setminus FQ$. Let $f_0=1$. We choose letters $a_1,a_2,\ldots$ inductively so that $a_{n+1}\ne a_n^{-1}$ and $f_nx\in D_{a_{n+1}}$, where $f_n=a_n^{-1}\cdots a_1^{-1}$. Since $x\notin Q$, we can choose $a_1$ with $x\in D_{a_1}$. Suppose that $a_1,\ldots,a_n$ have been chosen. By (b), we have $f_nx=a_n^{-1}f_{n-1}x\notin D_{a_n^{-1}}$. Also, $f_nx\notin Q$ since $x\notin FQ$. Hence, there exists $a_{n+1}\ne a_n^{-1}$ such that $f_nx\in D_{a_{n+1}}$. On the other hand, our assumption $1\in Q$ implies $f_n\in D_{a_n^{-1}}$. Thus, the pairwise distinct elements $f_n$, $n\ge1$, all belong to $E_x$, which contradicts finiteness of $E_x$.
\end{proof}

Let $\d_F$ be the word metric corresponding to the given basis of $F$, and let $\d_G$ be the word metric corresponding to a finite symmetric generating set $S$ of $G$ containing $x_1,\ldots,x_r$. We define $\pi\colon G\to F$ by letting $\pi(fq)=f$ for all $f\in F$ and $q\in Q$.

\begin{lem}\label{Lem:pi}
Let $\pi\colon G\to F$ be the map defined above.
\begin{enumerate}
\item[(a)] $\pi$ is an $F$-equivariant Lipschitz retraction.
\item[(b)] For every $R>0$, there exists $D_R>0$ such that for any elements $x,y\in G$ satisfying $\d_G(x,F)>D_R$ and $\d_G(x,y)\le R$, we have $\d_F(\pi(x),\pi(y))\le 1$.
\end{enumerate}
\end{lem}
\begin{proof}
Clearly, $\pi$ is an $F$-equivariant retraction. The ping-pong argument above also shows that $x\in D_a$ if and only if the reduced word representing $\pi(x)$ begins with $a$.

Fix any $x,g\in G$ and consider the segment $[\pi(x),\pi(xg)]$ in the Cayley tree of $F$. For every interior vertex $v$ of this segment, the words $v^{-1}\pi(x)$ and $v^{-1}\pi(xg)$ begin with distinct letters. By equivariance of $\pi$, we have $v^{-1}x\in E_g$. Distinct vertices give distinct elements of $E_g$. Hence,
$$
\d_F(\pi(x),\pi(xg))\le |E_g|+1,\qquad {\rm and}\qquad 
\d_F(\pi(x),\pi(xg))\ge2\ \Longrightarrow\ x\in FE_g.
$$
Taking $L=\max_{g\in S}(|E_g|+1)$ and applying the first estimate along edge paths, we conclude that $\pi$ is $L$-Lipschitz, thus completing the proof of  (a).

Further, for every $R>0$, we can choose $D_R>0$ such that $\d_G(1,z)\le D_R$ for all $z\in E_g$ and all $g\in G$ satisfying $\d_G(1,g)\le R$, since bounded balls in $G$ and the sets $E_g$ are finite. If $\d_G(x,y)\le R$, we can write $y=xg$, where $\d_G(1,g)\le R$. Taking into account that $x\in FE_g$ implies $\d_G(x,F)\le D_R$ for these $g$, we obtain (b).
\end{proof}

\begin{lem}\label{Lem:sqc}
The subgroup $F$ is strongly quasi-convex in $G$.
\end{lem}
\begin{proof}
Let $L$ be a Lipschitz constant of $\pi$ provided by Lemma \ref{Lem:pi}. For any $g\in G$, choosing a nearest point $f\in F$ to $g$ gives
\begin{equation}\label{Eq:dxpx}
\d_G(g,\pi(g))\le\d_G(g,f)+\d_F(f,\pi(g))=\d_G(g,f)+\d_F(\pi(f),\pi(g))
\le(L+1)\d_G(g,F).
\end{equation}

Fix any $\lambda\ge1$ and $c\ge0$. Let $p$ be a $(\lambda, c)$-quasi-geodesic in $Cay(G,S)$ with endpoints $p_-, p_+\in F$. Thus, for every subpath $q$ of $p$, we have $\ell (q)\le\lambda\d_G(p_-,p_+)+ c$. We choose integers $R>2\lambda$ and $\rho>D_R$. 

\begin{figure}
\centering\scalebox{.9}{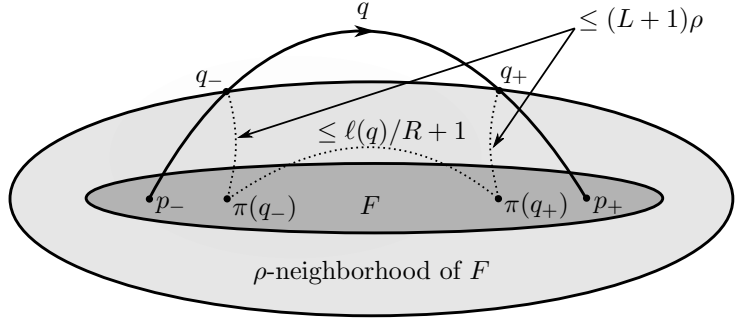}\\
  \caption{Estimating the Hausdorff distance from $p$ to $F$}
\end{figure}

Consider any subpath $q$ of $p$ whose endpoints have distance $\rho$ from $F$ and whose interior vertices have distance greater than $\rho$. Subdividing $q$ into at most $\ell(q)/R+1$ subpaths of length at most $R$ and using Lemma \ref{Lem:pi} (b), we obtain $\d_F(\pi(q_-),\pi(q_+))\le\ell(q)/R+1$. Using (\ref{Eq:dxpx}), we obtain
$$
\ell(q)\le\lambda\d_G(q_-,q_+)+ c
\le\lambda\Big(2(L+1)\rho+\ell(q)/R+1\Big) +c.
$$
Since $R>2\lambda$, we obtain $\ell(q)\le4\lambda(L+1)\rho+2\lambda +2c$. 

Observe that every vertex of $p$ outside the $\rho$-neighborhood of $F$ lies on a subpath $q$ as above. This results in a bound on the distance from $p$ to $F$ that depends on $\lambda$, $c$, and the fixed data only. Thus, $F$ is strongly quasi-convex.
\end{proof}

The proof of Proposition \ref{Prop:ping} is completed.

\paragraph{3.4. Lipschitz retractions to strongly quasi-convex free subgroups.} Our next goal is to construct Lipschitz projections to strongly quasi-convex free subgroups with strong contraction properties. We begin with an auxiliary metric construction.

\begin{lem}\label{Lem:d*}
Let $H$ be a strongly quasi-convex subgroup of a group $G$ generated by a finite set $X$, and let $\d$ be the word metric on $G$ with respect to $X$. There exists an $H$-invariant metric $\d^*$ on $G$ satisfying the following conditions.
\begin{enumerate}
\item[(a)] $\d^*\le\d$ and $\d^*|_{H\times H}=\d|_{H\times H}$.
\item[(b)] For every $R>0$, we have $\sup\{\d^*(x,y)\mid \d(x,y)\le R,\ \d(x,H)\ge D\}\to 0$ as $D\to\infty$.
\end{enumerate}
\end{lem}
\begin{proof}
For each $k\ge1$, choose $M_k>0$ such that every
$(2^k,0)$-quasi-geodesic with endpoints in $H$ lies in the
$M_k$-neighborhood of $H$. We choose increasing integers $r_k>M_k+2$.

For every edge $e$ of $Cay(G,X)$, set
$$
\nu(e)=\max\{k\ge0\mid
\min\{\d(e_-,H),\d(e_+,H)\}\ge r_k\}
$$
(we assume $r_0=0$ for convenience) and assign $e$ length $2^{-\nu(e)}$. Let $\d^*$ be the
resulting path metric on vertices. Clearly, $\d^*\le\d$,
and $\d^*$ is $H$-invariant as so are the edge lengths. 

Suppose that there exists a path $p$ from some $h\in H$ to some $h'\in H$ whose modified length is less than $\d(h,h')$. Choose such a path with the minimal possible number of edges. Let $k$ be the largest value of $\nu(e)$ among its edges. Necessarily, $k\ge1$ and the path $p$ is a $(2^k,0)$-quasi-geodesic with respect to $\d$. Indeed, if a subpath $q$ had original length greater than $2^k\d(q_-,q_+)$, replacing $q$ by a geodesic in $(Cay(G,X), \d)$ would strictly decrease both
the original and modified lengths: every edge of $p$ has
modified length at least $2^{-k}$, while every edge of the
replacement has modified length at most one. This would contradict the choice of $p$. Thus, $p$ is a $(2^k,0)$-quasi-geodesic and hence belongs to the $M_k$-neighborhood of $H$. However, $p$ contains an edge $e$ with $\nu(e)=k$, whose endpoints have distance at least $r_k>M_k$ from $H$.
 This contradiction proves (a).

Finally, fix $R>0$ and $k\ge1$. If $\d(x,H)\ge r_k+R$
and $\d(x,y)\le R$, every vertex of an original geodesic
from $x$ to $y$ has distance at least $r_k$ from $H$.
Every edge of this geodesic therefore has modified length
at most $2^{-k}$, giving
$$
\d^*(x,y)\le 2^{-k}\d(x,y)\le 2^{-k}R.
$$
Letting $k\to\infty$, we obtain (b).
\end{proof}

By the \textit{Cayley tree} $\mathcal T$ of a free group $F$ corresponding to a fixed free basis $X$ of $F$, we mean an unoriented, unlabeled graph with the vertex set $F$, where two vertices $u$ and $v$ are connected by an edge iff $u^{-1}v\in X^{\pm 1}$. We endow $\mathcal T$ with a metric $\d_{\mathcal T}$ by identifying every edge with the segment $[0,1]$.

\begin{prop}\label{Prop:contr}
Let $F$ be a strongly quasi-convex free subgroup of a finitely generated group $G$, let $\d$ be a word metric on $G$, and let $\mathcal T$ be a Cayley tree of $F$. There exists a map $\pi\colon G\to\mathcal T$ satisfying the following conditions.
\begin{enumerate}
\item[(a)] $\pi$ is $F$-equivariant and Lipschitz.
\item[(b)] $\pi\vert_F\equiv id_F$.
\item[(c)] For every $R>0$, we have  $\sup \{ \d_{\mathcal T}(\pi(x),\pi(y))\mid \d(x,y)\le R,\; \d(x,F)\ge D\} \to 0$ as $D\to\infty$.
\end{enumerate}
\end{prop}
\begin{proof}
Since $F$ is strongly quasi-convex, we can apply Lemma \ref{Lem:d*} with $H=F$. Let $\d^*$ be the metric provided by this lemma. Since $F$ is undistorted in $G$, there exists $L\ge1$ such that $\d_{\mathcal T}(f,f')\le L\d^*(f,f')$ for $f,f'\in F$. For $x\in G$, put
$$
K_x=\bigcap_{f\in F}\overline B_{\mathcal T}(f,L\d^*(x,f)),
$$
where $\overline B_{\mathcal T}(f,L\d^*(x,f))$ is the closed ball in $\mathcal T$ of radius $L\d^*(x,f)$ centered at $f$. Since $d_T(f,f')\le Ld^*(f,f')
\le Ld^*(x,f)+Ld^*(x,f')$ by the triangle inequality, these balls intersect pairwise. The Helly property for subtrees and compactness of closed balls in $\mathcal T$ show that $K_x$ is a nonempty compact subtree. Let $\pi(x)$ be its circumcenter. Clearly, $K_f=\{f\}$ and $K_{fx}=fK_x$ for all $f\in F$ and $x\in G$. Hence, $\pi$ is $F$-equivariant and condition (b) holds.

To show that $\pi$ is Lipschitz with respect to $\d^*$, we fix $x,y\in G$ and let $\e=L\d^*(x,y)$. If $z\in K_x$, then $\d_{\mathcal T}(z,f)\le L\d^*(y,f)+\e$ for every $f\in F$. Thus $\overline B_{\mathcal T}(z,\e)$ meets every ball defining $K_y$, and Helly's property gives $\d_{\mathcal T}(z,K_y)\le\e$. Interchanging $x,y$, we obtain $\d_H(K_x,K_y)\le\e$. Using the well-known fact that the distance between the circumcenters of two bounded subtrees does not exceed their Hausdorff distance, we obtain $\d_{\mathcal T}(\pi(x),\pi(y))\le\e$. Thus, $\pi$ is $L$-Lipschitz with respect to $\d^*$ and hence with respect to $\d$. This completes the proof of (a).

Finally, condition (c) follows from part (b) of Lemma \ref{Lem:d*}.
\end{proof}

\paragraph{3.5. Proof of Theorem \ref{Thm:AH} (b).}
To prove the forward implication, let $A_i\subseteq B_i$ and $g_i\in G$, $i=1,2$, satisfy Definition \ref{Def:PP}. Since the left side of (\ref{Eq:PP}) is at most $2$, we obtain the following (see Fig. \ref{Fig1}):
\begin{equation}\label{Eq:PP2}
B_1\cap B_2=\emptyset,\qquad G=g_1A_1\cup g_2A_2,\qquad
B_1\cup B_2\subseteq g_1A_1\cap g_2A_2.
\end{equation}

\begin{figure}
\centering\scalebox{.9}{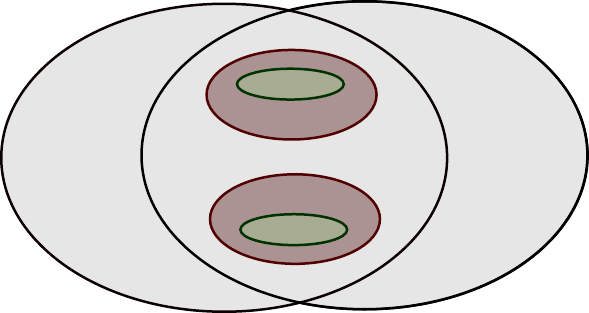}\\
  \caption{Relations between subsets in property $\PP(2)$}\label{Fig1}
\end{figure}

We set
$$
a_i=g_i^2,\qquad D_{a_i^{-1}}=A_i,\qquad D_{a_i}=G\setminus g_i^2A_i
\qquad(i=1,2).
$$
By (\ref{Eq:PP2}), the four domains are pairwise disjoint and their union misses the nonempty set $g_1A_2\subseteq g_1^2A_1\setminus g_1A_1$. Thus, condition (a) of Proposition \ref{Prop:ping} holds, while (b) follows immediately from the definitions.

To verify (c), observe that (\ref{Eq:PP2}) and the definitions give
$$
\begin{aligned}
D_{a_i^{-1}}h\cap D_c&\subseteq A_ih\setminus B_i
&&\text{if }c\ne a_i^{-1},\\
D_{a_i}h\cap D_c&\subseteq g_i(A_ih^{-1}\setminus B_i)h
&&\text{if }c\ne a_i
\end{aligned}
$$
for every $h\in G$, $i=1,2$, and $c\in\{a_1^{\pm1},a_2^{\pm1}\}$. Both right sides are finite by admissibility. Applying Proposition \ref{Prop:ping}, we conclude that $g_1^2,g_2^2$ freely generate a strongly quasi-convex subgroup of $G$.

To prove the backward implication, let $F$ be a strongly quasi-convex, non-cyclic, free subgroup of $G$. Fix a word metric $\d$ on $G$ and a finite basis $X$ of $F$ containing distinct elements $a,b$. Let $\mathcal T$ be the Cayley tree of $F$ with edges joining $f$ to $fc$ for $c\in X^{\pm1}$, so that $F$ acts on $\mathcal T$ by left multiplication. Let $\pi\colon G\to\mathcal T$ be the map provided by Proposition~\ref{Prop:contr}, and let $L$ denote its Lipschitz constant. We will pull back the decomposition from Example \ref{Ex:Free} (b) (see Fig. \ref{Fig0})  through $\pi$.

For $c\in\{a,b\}$, let $v_c$ be the midpoint of the edge $[1,c]$ in $\mathcal T$ and let $J_c$ be the closed half-tree containing $c$ and bounded by $v_c$. Thus, the vertices of $J_c$ and $cJ_c$ are precisely the reduced words beginning with $c$ and $c^2$, respectively. Set
$$
A_1=\pi^{-1}(aJ_a)\subseteq B_1=\pi^{-1}(J_a),\qquad
A_2=\pi^{-1}(bJ_b)\subseteq B_2=\pi^{-1}(J_b).
$$
Since $\pi\vert_F=id_F$, these sets restrict to the sets used in Example~1.2, with respect to the basis $X$.

We first verify admissibility. Fix $c\in\{a,b\}$ and $h\in G$, and suppose that $\pi(x)\in cJ_c$ and $\pi(xh)\notin J_c$. The distance between these two points is at least one, so Proposition \ref{Prop:contr} (c) gives $\d(x,F)\le D_h$ for a certain constant $D_h$ depending on $h$. The segment connecting them passes through $v_c$, and hence $\d_{\mathcal T}(\pi(x),v_c)\le L\d(1,h)$. Choosing $f\in F$ with $\d(x,f)\le D_h$, we obtain
$$
\d_{\mathcal T}(f,v_c)\le L(D_h+\d(1,h)).
$$
Since $\mathcal T$ has finite degree, only finitely many elements $f\in F$ satisfy this inequality. The $D_h$-neighborhood of this finite set in $G$ is also finite. Thus, there are only finitely many such $x$, and both pairs $A_i\subseteq B_i$ are admissible.

We take $g_1=a^{-2}$ and $g_2=b^{-2}$. The equivariance of $\pi$ gives $g_1A_1=\pi^{-1}(a^{-1}J_a)$ and $g_2A_2=\pi^{-1}(b^{-1}J_b)$. Hence, the pairs above satisfy (\ref{Eq:PP2}); indeed, the corresponding relations are obvious for the images of our sets under $\pi$ and persist under passing to preimages. Therefore
$$
\one_{a^{-2}A_1}+\one_{b^{-2}A_2}
\ge\one_G+\one_{B_1}+\one_{B_2}.
$$

\section{Some examples}

The goal of this section is to discuss properly proximal groups of PP-complexity greater than $2$. 

\paragraph{4.1. Groups of PP-complexity $3$.} We begin with a very simple example. 

\begin{prop}\label{Lem:F2xF2}
For a finitely generated free group $F$ of rank at least $2$, the PP-complexity of $F\times F$ is $3$.
\end{prop}

\begin{proof}
Let $G=F\times F$ and let $\{ a, b, \ldots \}$ be a basis in $F$. Let $U\subseteq F$ consist of the reduced words in $F$ not beginning with $a^{-1}$, including
the identity. Since right multiplication by a fixed word
changes membership in $U$ only for finitely many words,
the pair $U\subseteq U$ is admissible in $F$.

Let
$$
A=U\times U,\qquad
B=(U\times F)\cup(F\times U).
$$
For every $(r,s)\in G$, we have $A(r,s)\setminus B=(Ur\setminus U)\times(Us\setminus U)$. Thus, the pair $A\subseteq B$ is admissible. 

For $i=0,1,2$, set
$$
u_i=b^i,\qquad v_i=ab^ia,\qquad
s_i=(u_i,u_i),\qquad p_i=(v_i,v_i).
$$
The six sets $u_i(F\setminus U)$ and $v_iU$, $i=0,1,2$,
are pairwise disjoint as their elements begin with
$a^{-1},ba^{-1},b^2a^{-1},a^2,aba,ab^2a$, respectively. Each coordinate of a point $(x,y)\in G$ belongs to at most
one of these six sets. Hence $(x,y)$ belongs to at most
two of the six sets $G\setminus s_iA$ and $p_iB$.
Therefore,
$$
\sum_{i=0}^2\bigl(1-\one_{s_iA}\bigr)
+\sum_{i=0}^2\one_{p_iB}\le 2,
$$
which gives $\PP(3)$.  On the other hand, $G$ contains no Morse elements, so it does not satisfy $\PP(2)$ by Theorem \ref{Thm:AH} (b).
\end{proof}

\paragraph{4.2. Groups of arbitrarily large PP-complexity.} Our next goal is to establish the existence of properly proximal groups of arbitrarily large PP-complexity. Our approach is motivated by results of Ershov--Golan--Sapir on Tarski numbers. We begin with an analogue of \cite[Lemma 4.1]{EGS}.

\begin{lem}\label{Lem:subgi}
    Let $G$ be a group. Suppose that $A_i\subseteq B_i\subseteq G$ and $g_i\in G$, $1\le i\le n$, are subsets and elements satisfying Definition \ref{Def:PP}. Then the subgroup $H=\langle g_1, \ldots, g_{n}\rangle$ is infinite.
\end{lem}

\begin{proof}
Suppose that $H$ is finite. Since $A_i\subseteq B_i$, summing (\ref{Eq:PP}) over $H$ gives
\begin{equation}\label{Eq:sum}
\sum_{i=1}^n\sum_{x\in H}
\Big(\one_{g_iA_i}(x)-\one_{A_i}(x)\Big)\ge \sum_{i=1}^n\sum_{x\in H}
\Big(\one_{g_iA_i}(x)-\one_{B_i}(x)\Big)\ge |H|.
\end{equation}
However, $g_i^{-1}H=H$ for every $i$, so
$$
\sum_{x\in H}\one_{g_iA_i}(x)
=\sum_{x\in H}\one_{A_i}(g_i^{-1}x)
=\sum_{x\in H}\one_{A_i}(x).
$$
Thus, the left side of (\ref{Eq:sum}) equals $0$, a contradiction.
\end{proof}

\begin{proof}[Proof of Theorem \ref{Thm:FI}]
Building on work of Golod \cite{Gol} (see also \cite[Theorem 3.3]{Ers11}), Ershov--Golan--Sapir \cite[Proposition 4.6]{EGS} constructed a finitely generated Golod--Shafarevich group $G$ with the following property:

\smallskip

($\ast$) \textit{For every integer $m\ge 1$, there exists a finite index subgroup $H_m$ of $G$ such that every $m$-generated subgroup of $H_m$ is finite.}

\smallskip

We refer the interested reader to \cite{Ers11} for the definition and basic properties of Golod--Shafarevich groups. Here we only need the fact that every such group is non-amenable \cite{Ers12}. Thus, the wreath product group $W=\ZZ_2\wr G$ is properly proximal by \cite[Theorem~1.7]{DKE}. Moreover, $W$ is finitely generated.

Fix an integer $m\ge 2$, and let $K_m$ be the full preimage of $H_m$ under the canonical epimorphism $\e\colon W\to G$.  Then $[W:K_m]=[G:H_m]<\infty$. Since finite index subgroups of properly proximal groups are properly proximal \cite[Proposition 4.10]{BIP}, every $K_m$ is properly proximal.

We claim that the PP-complexity of $K_m$ is greater than $m$. Suppose, to the contrary, that $K_m$ satisfies $\PP(m)$, and let $g_1,\ldots,g_m\in K_m$ be translations witnessing $\PP(m)$. The image of $L_m=\langle g_1,\ldots,g_m\rangle$ under $\e$ is finite by ($\ast$). Hence $L_m\cap\Ker\e$ has finite index in
$L_m$ and is therefore finitely generated. Since $\Ker\e=\bigoplus_{g\in G}\ZZ_2$ is locally finite, the subgroup $L_m\cap\Ker\e$ is finite. Thus, $L_m$ is finite, contradicting Lemma \ref{Lem:subgi}.
\end{proof}

\end{document}

%% file: f0.pdf_tex
\begingroup%
  \makeatletter%
  \providecommand\color[2][]{%
    \errmessage{(Inkscape) Color is used for the text in Inkscape, but the package 'color.sty' is not loaded}%
    \renewcommand\color[2][]{}%
  }%
  \providecommand\transparent[1]{%
    \errmessage{(Inkscape) Transparency is used (non-zero) for the text in Inkscape, but the package 'transparent.sty' is not loaded}%
    \renewcommand\transparent[1]{}%
  }%
  \providecommand\rotatebox[2]{#2}%
  \newcommand*\fsize{\dimexpr\f@size pt\relax}%
  \newcommand*\lineheight[1]{\fontsize{\fsize}{#1\fsize}\selectfont}%
  \ifx\svgwidth\undefined%
    \setlength{\unitlength}{540.95314287bp}%
    \ifx\svgscale\undefined%
      \relax%
    \else%
      \setlength{\unitlength}{\unitlength * \real{\svgscale}}%
    \fi%
  \else%
    \setlength{\unitlength}{\svgwidth}%
  \fi%
  \global\let\svgwidth\undefined%
  \global\let\svgscale\undefined%
  \makeatother%
  \begin{picture}(1,0.41166225)%
    \lineheight{1}%
    \setlength\tabcolsep{0pt}%
    \put(0,0){\includegraphics[width=\unitlength,page=1]{f0.pdf}}%
    \put(0.28712461,0.31975182){\color[rgb]{0,0,0}\makebox(0,0)[lt]{\begin{minipage}{0.18773406\unitlength}\raggedright \end{minipage}}}%
    \put(0.34131307,0.21420728){\color[rgb]{0,0,0}\makebox(0,0)[lt]{\lineheight{1.25}\smash{\begin{tabular}[t]{l}$A_1$\end{tabular}}}}%
    \put(0.19402924,0.36247345){\color[rgb]{0,0,0}\makebox(0,0)[lt]{\lineheight{1.25}\smash{\begin{tabular}[t]{l}$A_2$\end{tabular}}}}%
    \put(0.23545662,0.33143378){\color[rgb]{0,0,0}\makebox(0,0)[lt]{\lineheight{1.25}\smash{\begin{tabular}[t]{l}$B_2$\end{tabular}}}}%
    \put(0.30826336,0.25984117){\color[rgb]{0,0,0}\makebox(0,0)[lt]{\lineheight{1.25}\smash{\begin{tabular}[t]{l}$B_1$\end{tabular}}}}%
    \put(0,0){\includegraphics[width=\unitlength,page=2]{f0.pdf}}%
    \put(0.14174091,0.20476568){\color[rgb]{0,0,0}\makebox(0,0)[lt]{\lineheight{1.25}\smash{\begin{tabular}[t]{l}$a$\end{tabular}}}}%
    \put(0.20642873,0.20476568){\color[rgb]{0,0,0}\makebox(0,0)[lt]{\lineheight{1.25}\smash{\begin{tabular}[t]{l}$a$\end{tabular}}}}%
    \put(0.18666303,0.22042424){\color[rgb]{0,0,0}\makebox(0,0)[lt]{\lineheight{1.25}\smash{\begin{tabular}[t]{l}$b$\end{tabular}}}}%
    \put(0.1869197,0.15933021){\color[rgb]{0,0,0}\makebox(0,0)[lt]{\lineheight{1.25}\smash{\begin{tabular}[t]{l}$b$\end{tabular}}}}%
    \put(0,0){\includegraphics[width=\unitlength,page=3]{f0.pdf}}%
    \put(0.66304302,0.34047507){\color[rgb]{0,0,0}\makebox(0,0)[lt]{\lineheight{1.25}\smash{\begin{tabular}[t]{l}$B_2$\end{tabular}}}}%
    \put(0,0){\includegraphics[width=\unitlength,page=4]{f0.pdf}}%
    \put(0.58540205,0.20476572){\color[rgb]{0,0,0}\makebox(0,0)[lt]{\lineheight{1.25}\smash{\begin{tabular}[t]{l}$a$\end{tabular}}}}%
    \put(0.65008886,0.20476572){\color[rgb]{0,0,0}\makebox(0,0)[lt]{\lineheight{1.25}\smash{\begin{tabular}[t]{l}$a$\end{tabular}}}}%
    \put(0.63032353,0.22042428){\color[rgb]{0,0,0}\makebox(0,0)[lt]{\lineheight{1.25}\smash{\begin{tabular}[t]{l}$b$\end{tabular}}}}%
    \put(0.63058011,0.15933027){\color[rgb]{0,0,0}\makebox(0,0)[lt]{\lineheight{1.25}\smash{\begin{tabular}[t]{l}$b$\end{tabular}}}}%
    \put(0,0){\includegraphics[width=\unitlength,page=5]{f0.pdf}}%
    \put(0.76643142,0.24273118){\color[rgb]{0,0,0}\makebox(0,0)[lt]{\lineheight{1.25}\smash{\begin{tabular}[t]{l}$B_1$\end{tabular}}}}%
    \put(0.67178269,0.04469837){\color[rgb]{0,0,0}\makebox(0,0)[lt]{\lineheight{1.25}\smash{\begin{tabular}[t]{l}$a^{-2}A_1$\end{tabular}}}}%
    \put(0.77441438,0.34434825){\color[rgb]{0,0,0}\makebox(0,0)[lt]{\lineheight{1.25}\smash{\begin{tabular}[t]{l}$a^{-2}A_1\cap b^{-2}A_2$\end{tabular}}}}%
    \put(0.45619329,0.28261132){\color[rgb]{0,0,0}\makebox(0,0)[lt]{\lineheight{1.25}\smash{\begin{tabular}[t]{l}$b^{-2}A_2$\end{tabular}}}}%
    \put(0,0){\includegraphics[width=\unitlength,page=6]{f0.pdf}}%
  \end{picture}%
\endgroup%

%% file: f2.pdf_tex
\begingroup%
  \makeatletter%
  \providecommand\color[2][]{%
    \errmessage{(Inkscape) Color is used for the text in Inkscape, but the package 'color.sty' is not loaded}%
    \renewcommand\color[2][]{}%
  }%
  \providecommand\transparent[1]{%
    \errmessage{(Inkscape) Transparency is used (non-zero) for the text in Inkscape, but the package 'transparent.sty' is not loaded}%
    \renewcommand\transparent[1]{}%
  }%
  \providecommand\rotatebox[2]{#2}%
  \newcommand*\fsize{\dimexpr\f@size pt\relax}%
  \newcommand*\lineheight[1]{\fontsize{\fsize}{#1\fsize}\selectfont}%
  \ifx\svgwidth\undefined%
    \setlength{\unitlength}{227.57080391bp}%
    \ifx\svgscale\undefined%
      \relax%
    \else%
      \setlength{\unitlength}{\unitlength * \real{\svgscale}}%
    \fi%
  \else%
    \setlength{\unitlength}{\svgwidth}%
  \fi%
  \global\let\svgwidth\undefined%
  \global\let\svgscale\undefined%
  \makeatother%
  \begin{picture}(1,0.53848422)%
    \lineheight{1}%
    \setlength\tabcolsep{0pt}%
    \put(0,0){\includegraphics[width=\unitlength,page=1]{f2.pdf}}%
    \put(0.28784281,0.09135367){\color[rgb]{0,0,0}\makebox(0,0)[lt]{\lineheight{1.25}\smash{\begin{tabular}[t]{l}$h_i^r$\end{tabular}}}}%
    \put(0.67245506,0.09132254){\color[rgb]{0,0,0}\makebox(0,0)[lt]{\lineheight{1.25}\smash{\begin{tabular}[t]{l}$h_i^s$\end{tabular}}}}%
    \put(0.89580224,0.5114322){\color[rgb]{0,0,0}\makebox(0,0)[lt]{\lineheight{1.25}\smash{\begin{tabular}[t]{l}$g$\end{tabular}}}}%
    \put(0.12205115,0.12376128){\color[rgb]{0,0,0}\makebox(0,0)[lt]{\lineheight{1.25}\smash{\begin{tabular}[t]{l}$H_i$\end{tabular}}}}%
    \put(0,0){\includegraphics[width=\unitlength,page=2]{f2.pdf}}%
    \put(0.58484001,0.36308594){\color[rgb]{0,0,0}\makebox(0,0)[lt]{\lineheight{1.25}\smash{\begin{tabular}[t]{l}$p$\end{tabular}}}}%
    \put(0.78334112,0.28300392){\color[rgb]{0,0,0}\makebox(0,0)[lt]{\lineheight{1.25}\smash{\begin{tabular}[t]{l}$q$\end{tabular}}}}%
    \put(0.4936829,0.16628872){\color[rgb]{0,0,0}\makebox(0,0)[lt]{\lineheight{1.25}\smash{\begin{tabular}[t]{l}$e$\end{tabular}}}}%
  \end{picture}%
\endgroup%

%% file: f3.pdf_tex
\begingroup%
  \makeatletter%
  \providecommand\color[2][]{%
    \errmessage{(Inkscape) Color is used for the text in Inkscape, but the package 'color.sty' is not loaded}%
    \renewcommand\color[2][]{}%
  }%
  \providecommand\transparent[1]{%
    \errmessage{(Inkscape) Transparency is used (non-zero) for the text in Inkscape, but the package 'transparent.sty' is not loaded}%
    \renewcommand\transparent[1]{}%
  }%
  \providecommand\rotatebox[2]{#2}%
  \newcommand*\fsize{\dimexpr\f@size pt\relax}%
  \newcommand*\lineheight[1]{\fontsize{\fsize}{#1\fsize}\selectfont}%
  \ifx\svgwidth\undefined%
    \setlength{\unitlength}{459.30576907bp}%
    \ifx\svgscale\undefined%
      \relax%
    \else%
      \setlength{\unitlength}{\unitlength * \real{\svgscale}}%
    \fi%
  \else%
    \setlength{\unitlength}{\svgwidth}%
  \fi%
  \global\let\svgwidth\undefined%
  \global\let\svgscale\undefined%
  \makeatother%
  \begin{picture}(1,0.2835879)%
    \lineheight{1}%
    \setlength\tabcolsep{0pt}%
    \put(0,0){\includegraphics[width=\unitlength,page=1]{f3.pdf}}%
    \put(0.13183769,0.03873107){\color[rgb]{0,0,0}\makebox(0,0)[lt]{\lineheight{1.25}\smash{\begin{tabular}[t]{l}$a$\end{tabular}}}}%
    \put(0.30786514,0.03871564){\color[rgb]{0,0,0}\makebox(0,0)[lt]{\lineheight{1.25}\smash{\begin{tabular}[t]{l}$b$\end{tabular}}}}%
    \put(0.05595867,0.06131958){\color[rgb]{0,0,0}\makebox(0,0)[lt]{\lineheight{1.25}\smash{\begin{tabular}[t]{l}$H_i$\end{tabular}}}}%
    \put(0,0){\includegraphics[width=\unitlength,page=2]{f3.pdf}}%
    \put(0.21951421,0.03993523){\color[rgb]{0,0,0}\makebox(0,0)[lt]{\lineheight{1.25}\smash{\begin{tabular}[t]{l}$e$\end{tabular}}}}%
    \put(0.32005158,0.25128469){\color[rgb]{0,0,0}\makebox(0,0)[lt]{\lineheight{1.25}\smash{\begin{tabular}[t]{l}$gf$\end{tabular}}}}%
    \put(0.11576475,0.25205693){\color[rgb]{0,0,0}\makebox(0,0)[lt]{\lineheight{1.25}\smash{\begin{tabular}[t]{l}$g$\end{tabular}}}}%
    \put(0,0){\includegraphics[width=\unitlength,page=3]{f3.pdf}}%
    \put(0.21579445,0.27018451){\color[rgb]{0,0,0}\makebox(0,0)[lt]{\lineheight{1.25}\smash{\begin{tabular}[t]{l}$gr$\end{tabular}}}}%
    \put(0,0){\includegraphics[width=\unitlength,page=4]{f3.pdf}}%
    \put(0.6739594,0.03873109){\color[rgb]{0,0,0}\makebox(0,0)[lt]{\lineheight{1.25}\smash{\begin{tabular}[t]{l}$a$\end{tabular}}}}%
    \put(0.84998886,0.03871566){\color[rgb]{0,0,0}\makebox(0,0)[lt]{\lineheight{1.25}\smash{\begin{tabular}[t]{l}$b$\end{tabular}}}}%
    \put(0.5980799,0.0613196){\color[rgb]{0,0,0}\makebox(0,0)[lt]{\lineheight{1.25}\smash{\begin{tabular}[t]{l}$H_i$\end{tabular}}}}%
    \put(0,0){\includegraphics[width=\unitlength,page=5]{f3.pdf}}%
    \put(0.76163671,0.03993526){\color[rgb]{0,0,0}\makebox(0,0)[lt]{\lineheight{1.25}\smash{\begin{tabular}[t]{l}$e$\end{tabular}}}}%
    \put(0.86217048,0.25128471){\color[rgb]{0,0,0}\makebox(0,0)[lt]{\lineheight{1.25}\smash{\begin{tabular}[t]{l}$gf$\end{tabular}}}}%
    \put(0.65788675,0.25205693){\color[rgb]{0,0,0}\makebox(0,0)[lt]{\lineheight{1.25}\smash{\begin{tabular}[t]{l}$g$\end{tabular}}}}%
    \put(0,0){\includegraphics[width=\unitlength,page=6]{f3.pdf}}%
    \put(0.77447858,0.18392942){\color[rgb]{0,0,0}\makebox(0,0)[lt]{\lineheight{1.25}\smash{\begin{tabular}[t]{l}$gr$\end{tabular}}}}%
    \put(0,0){\includegraphics[width=\unitlength,page=7]{f3.pdf}}%
    \put(0.7606508,0.09175418){\color[rgb]{0,0,0}\makebox(0,0)[lt]{\lineheight{1.25}\smash{\begin{tabular}[t]{l}$gv=h_i^m$\end{tabular}}}}%
  \end{picture}%
\endgroup%

%% file: f4.pdf_tex
\begingroup%
  \makeatletter%
  \providecommand\color[2][]{%
    \errmessage{(Inkscape) Color is used for the text in Inkscape, but the package 'color.sty' is not loaded}%
    \renewcommand\color[2][]{}%
  }%
  \providecommand\transparent[1]{%
    \errmessage{(Inkscape) Transparency is used (non-zero) for the text in Inkscape, but the package 'transparent.sty' is not loaded}%
    \renewcommand\transparent[1]{}%
  }%
  \providecommand\rotatebox[2]{#2}%
  \newcommand*\fsize{\dimexpr\f@size pt\relax}%
  \newcommand*\lineheight[1]{\fontsize{\fsize}{#1\fsize}\selectfont}%
  \ifx\svgwidth\undefined%
    \setlength{\unitlength}{302.38032376bp}%
    \ifx\svgscale\undefined%
      \relax%
    \else%
      \setlength{\unitlength}{\unitlength * \real{\svgscale}}%
    \fi%
  \else%
    \setlength{\unitlength}{\svgwidth}%
  \fi%
  \global\let\svgwidth\undefined%
  \global\let\svgscale\undefined%
  \makeatother%
  \begin{picture}(1,0.43722442)%
    \lineheight{1}%
    \setlength\tabcolsep{0pt}%
    \put(0,0){\includegraphics[width=\unitlength,page=1]{f4.pdf}}%
    \put(0.33712722,0.05279791){\color[rgb]{0,0,0}\makebox(0,0)[lt]{\lineheight{1.25}\smash{\begin{tabular}[t]{l}$\rho$-neighborhood of $F$\end{tabular}}}}%
    \put(0,0){\includegraphics[width=\unitlength,page=2]{f4.pdf}}%
    \put(0.20318608,0.14987977){\color[rgb]{0,0,0}\makebox(0,0)[lt]{\lineheight{1.25}\smash{\begin{tabular}[t]{l}$p_-$\end{tabular}}}}%
    \put(0.80660447,0.15059627){\color[rgb]{0,0,0}\makebox(0,0)[lt]{\lineheight{1.25}\smash{\begin{tabular}[t]{l}$p_+$\end{tabular}}}}%
    \put(0.48402193,0.14164151){\color[rgb]{0,0,0}\makebox(0,0)[lt]{\lineheight{1.25}\smash{\begin{tabular}[t]{l}$F$\end{tabular}}}}%
    \put(0,0){\includegraphics[width=\unitlength,page=3]{f4.pdf}}%
    \put(0.68296633,0.14495424){\color[rgb]{0,0,0}\makebox(0,0)[lt]{\lineheight{1.25}\smash{\begin{tabular}[t]{l}$\pi(q_+)$\end{tabular}}}}%
    \put(0.30840041,0.14079563){\color[rgb]{0,0,0}\makebox(0,0)[lt]{\lineheight{1.25}\smash{\begin{tabular}[t]{l}$\pi(q_-)$\end{tabular}}}}%
    \put(0,0){\includegraphics[width=\unitlength,page=4]{f4.pdf}}%
    \put(0.78546413,0.39957628){\color[rgb]{0,0,0}\makebox(0,0)[lt]{\lineheight{1.25}\smash{\begin{tabular}[t]{l}$\le (L+1)\rho$\end{tabular}}}}%
    \put(0.4799558,0.41686512){\color[rgb]{0,0,0}\makebox(0,0)[lt]{\lineheight{1.25}\smash{\begin{tabular}[t]{l}$q$\end{tabular}}}}%
    \put(0.25769775,0.32318303){\color[rgb]{0,0,0}\makebox(0,0)[lt]{\lineheight{1.25}\smash{\begin{tabular}[t]{l}$q_-$\end{tabular}}}}%
    \put(0.67730792,0.32876781){\color[rgb]{0,0,0}\makebox(0,0)[lt]{\lineheight{1.25}\smash{\begin{tabular}[t]{l}$q_+$\end{tabular}}}}%
    \put(0.42532384,0.24589527){\color[rgb]{0,0,0}\makebox(0,0)[lt]{\lineheight{1.25}\smash{\begin{tabular}[t]{l}$\le \ell(q)/R+1$\end{tabular}}}}%
    \put(0,0){\includegraphics[width=\unitlength,page=5]{f4.pdf}}%
  \end{picture}%
\endgroup%

%% file: f1.pdf_tex
\begingroup%
  \makeatletter%
  \providecommand\color[2][]{%
    \errmessage{(Inkscape) Color is used for the text in Inkscape, but the package 'color.sty' is not loaded}%
    \renewcommand\color[2][]{}%
  }%
  \providecommand\transparent[1]{%
    \errmessage{(Inkscape) Transparency is used (non-zero) for the text in Inkscape, but the package 'transparent.sty' is not loaded}%
    \renewcommand\transparent[1]{}%
  }%
  \providecommand\rotatebox[2]{#2}%
  \newcommand*\fsize{\dimexpr\f@size pt\relax}%
  \newcommand*\lineheight[1]{\fontsize{\fsize}{#1\fsize}\selectfont}%
  \ifx\svgwidth\undefined%
    \setlength{\unitlength}{282.64061048bp}%
    \ifx\svgscale\undefined%
      \relax%
    \else%
      \setlength{\unitlength}{\unitlength * \real{\svgscale}}%
    \fi%
  \else%
    \setlength{\unitlength}{\svgwidth}%
  \fi%
  \global\let\svgwidth\undefined%
  \global\let\svgscale\undefined%
  \makeatother%
  \begin{picture}(1,0.53152261)%
    \lineheight{1}%
    \setlength\tabcolsep{0pt}%
    \put(0,0){\includegraphics[width=\unitlength,page=1]{f1.pdf}}%
    \put(0.09124523,0.25013303){\color[rgb]{0,0,0}\makebox(0,0)[lt]{\lineheight{1.25}\smash{\begin{tabular}[t]{l}$g_1A_1$\end{tabular}}}}%
    \put(4.64713585,-0.88921865){\color[rgb]{0,0,0}\makebox(0,0)[lt]{\begin{minipage}{0.03186631\unitlength}\raggedright \end{minipage}}}%
    \put(0.85667003,0.25013307){\color[rgb]{0,0,0}\makebox(0,0)[lt]{\lineheight{1.25}\smash{\begin{tabular}[t]{l}$g_2A_2$\end{tabular}}}}%
    \put(0.48025256,0.32039777){\color[rgb]{0,0,0}\makebox(0,0)[lt]{\lineheight{1.25}\smash{\begin{tabular}[t]{l}$B_1$\end{tabular}}}}%
    \put(0.48025263,0.37730817){\color[rgb]{0,0,0}\makebox(0,0)[lt]{\lineheight{1.25}\smash{\begin{tabular}[t]{l}$A_1$\end{tabular}}}}%
    \put(0.48025256,0.18972398){\color[rgb]{0,0,0}\makebox(0,0)[lt]{\lineheight{1.25}\smash{\begin{tabular}[t]{l}$B_2$\end{tabular}}}}%
    \put(0.48025256,0.12983773){\color[rgb]{0,0,0}\makebox(0,0)[lt]{\lineheight{1.25}\smash{\begin{tabular}[t]{l}$A_2$\end{tabular}}}}%
  \end{picture}%
\endgroup%